\documentclass{amsart}
\usepackage{amsmath,amssymb,mathrsfs,bbm,yhmath,longtable,mathtools}
\usepackage{amsthm,rotating,xcolor,epsfig,hyperref,url,MnSymbol,rotate}
\usepackage{graphicx}
\usepackage[weather]{ifsym}

\usepackage{caption}
\usepackage{subcaption}

\usepackage{tikz}
\usetikzlibrary{knots}
\usetikzlibrary{decorations.pathmorphing,decorations.markings}
\usetikzlibrary{arrows,positioning}
\usetikzlibrary{svg.path}

\theoremstyle{plain}
\newtheorem{lemma}{Lemma}
\newtheorem{theorem}{Theorem}

\newcommand{\Int}{{\mbox{Int\ }}}

\DeclareMathOperator{\sign}{sign}

\newcommand{\crn}{\operatorname{cr}}

\newcommand{\cross}[1]{\#{#1}}

\newcommand{\PairPlus}[1]{\#\#{#1^{+}}}
\newcommand{\PairMinus}[1]{\#\#{#1^{-}}}
\newcommand{\PairPlusMinus}[1]{\#\#{#1^{\pm}}}

\newcommand{\CPlus}{{C^+}}
\newcommand{\CMinus}{{C^-}}
\newcommand{\CPlusMinus}{C^\pm}
\newcommand{\CMinusPlus}{C^\mp}

\newcommand{\CHPlus}{CH^+}
\newcommand{\CHMinus}{CH^-}
\newcommand{\CHPlusMinus}{CH^\pm}

\newcommand{\Plus}[1]{{{\scriptstyle +}\,\!#1}}
\newcommand{\Minus}[1]{{{\scriptstyle -}\,\!#1}}
\newcommand{\PlusMinus}[1]{{{\scriptstyle \pm}\,\!#1}}
\newcommand{\MinusPlus}[1]{{{\scriptstyle \mp}\,\!#1}}

\DeclareMathOperator{\lk}{lk}
\newcommand{\lkPlus}{{\lk^{+}}}
\newcommand{\lkMinus}{{\lk^{-}}}

\newcommand{\Angle}[1]{\langle{#1}\rangle}

\newcommand{\ModM}{\mathcal{M}}

\title{Homological Casson type invariant of knotoids}
\author{Vladimir Tarkaev}
\address{\noindent St.\ Petersburg State University, St.\ Petersburg, Russia}
 \address{\noindent Chelyabinsk State University, Chelyabinsk, Russia}
\address{\noindent Krasovskii Institute of Mathematics and Mechanics, Ural branch of the Russian Academy of Sciences, Ekaterinburg, Russia}
\thanks{The work is supported by the Russian Science Foundation
under grant~19-11-00151.}
\email{v.tarkaev@spbu.ru}

\begin{document}

\begin{abstract}
We consider an analogue of well-known Casson knot invariant for knotoids.
We start  with a direct analogue of the classical construction
which  gives two different integer-valued knotoid invariants
and then focus on its homology extension.
Value of the extension is a formal sum of subgroups of the first homology group $H_1(\Sigma)$
where $\Sigma$ is an oriented surface with (maybe) non-empty boundary in which knotoid diagrams lie.
To make the extension informative for spherical knotoids
it is sufficient to transform an initial knotoid diagram in $S^2$ into a knotoid diagram in the annulus
by removing small disks around its endpoints.
As an application of the invariants we prove two theorems:
a sharp lower bound of the crossing number of a knotoid
(the estimate differs from its prototype for classical knots proved by M.\,Polyak and O.\,Viro in 2001)
and a sufficient condition for a knotoid in $S^2$
to be a proper knotoid (or pure knotoid with respect to Turaev's terminology).
Finally we give a table containing values of our invariants
computed for all spherical prime proper knotoids having diagrams with at most $5$ crossings.
\end{abstract}

\maketitle

\section*{Introduction}
\label{section:Introduction}  

The concept of knotoid is introduced by V.\,Turaev~\cite{TuraevKnotoids}.
Later the subject was investigated by a few groups of researchers.
For a survey of existing works in the area
see~\cite{GKL}.
 For comprehensive tables of knotoids
see~\cite{GDS,KorablevMayTarkaev}.

Intuitively, knotoids can be considered as open-ended knot-type pictures up to an appropriate equivalence.
More precisely. Let $\Sigma$ be an oriented surface with (maybe) non-empty boundary.
Knotoid diagrams are generic immersions of the unit interval into $\Sigma$,
together with the under/over-crossing information at double points.
Knotoids are defined as the equivalence classes of knotoid diagrams under isotopies and the Reidemeister moves
(precise definitions are given in Section~\ref{sec:Preliminaries}).
In~\cite{TuraevKnotoids} Turaev shows
that knotoids in $S^2$ generalize knots in $S^3$
and that knotoids are closely related to knots in thickened surfaces via the closure operation.
Later in~\cite{KauffmanInvariants}
Kauffman and Gugumcu introduced and studied virtual knotoids
which generalize classical knotoids likewise virtual knots generalize classical knots.

Since knotoids are defined via knot-type diagrams
any diagram based invariant of knots (equally classical and virtual)
can be with appropriate changes extended to knotoids
(e.g., the knotoid group and the Kauffman bracket polynomial \cite{TuraevKnotoids},
the arrow polynomial and the affine index polynomial \cite{KauffmanInvariants}).
But unlike a knot diagram a knotoid diagram has endpoints.
The fact sometimes leads to pure knotoid features of such an extension.
For example, mentioned above Turaev's extension of the Kauffman bracket polynomial
has an additional indeterminate  counting intersections of an arc connecting the endpoints
with the rest part of a diagram.
In~\cite{Kutluay} Kutluay  defines winding homology --- a Khovanov-type invariant of knotoids categorifying this polynomial.
An additional grading in winding homology corresponds with additional indeterminate in the extended bracket polynomial.
These improvements makes the polynomial and winding homology
more informative than direct analogues of the Kauffman bracket polynomial and Khovanov homology
which are also studied in~\cite{TuraevKnotoids} and~\cite{Kutluay}.

 In present paper we deal with  similar situation:
we consider another well-known invariant of classical knots,
follow~\cite{PolyakViro2001} we call it the Casson knot invariant.
Its construction is not applicable to knots in thickened surfaces whenever the surface is not $S^2$
(see Remark in the end of Section~\ref{sec:K2_1})
but in the case of knotoid diagrams it works without any changes.
However, the nature of knotoids enables a pure knotoid strengthening of the extension.

The Casson knot invariant (see~\cite{PolyakViro2001})
can be defined in the terms of  a specific alternation of overpasses and underpasses in pairs of crossings.
We call such a configuration the skew pair of crossings
(for precise definition see Section~\ref{sec:SkewPairs}).
For each skew pair we define the sign ($\pm 1$),
the value of the invariant is the sum of these numbers over all skew pairs.
The same (word-for-word) gives an invariant of a knotoid.
Moreover, in the classical case it is necessary to prove the independence of resulting value on the choice of base point,
while in the case of knotoid diagram we have no freedom in the choice
because no point except the beginning (and points lying in its small neighborhood) can play the role.
That makes the situation simpler than its classical prototype.

At the same time the Casson-type invariant of knotoid obtains some new properties.
First of all, it admits a homological extension.
The idea is following
(precise definition is given in Section~\ref{sec:CHPlusMinus}).
There is a natural way to associate with a crossing $x$ in a knotoid diagram $D \subset \Sigma$ a loop $l(x)$ in the surface $\Sigma$.
Then we can associate with a skew pair $(x,y)$  of crossings the subgroup $H(x,y)$
of the first homology group $H_1(\Sigma)$
generated by the homology classes $[l(x)]$ and $[l(y)]$.
Finally we take a formal sum over all skew pairs of such a subgroups multiplying by the signs of corresponding pairs.
The resulting value $\CHPlusMinus$ is an invariant of a knotoid up to automorphism of the group $H_1(\Sigma)$
(see Theorem~\ref{theorem:CHPlusMinus}).
The direct (integer-valued) analog of the Casson knot invariant is also an invariant of a knotoid
but it is weaker than the homological extenntion
(see Section~\ref{sec:Table}).
Note that $\CHPlusMinus$ also makes sense in the case  of knotoids in the $2$-sphere.
To this end it is sufficient  to transform  initial  knotoid into a knotoid in the annulus (see Section~\ref{sec:InS^2})
by removing small disks around the endpoints.
The first homology group of the annulus is non-trivial
that allows  the invariant to be informative
(see Sections~\ref{sec:Example}
and~\ref{sec:Table}).

As an application we prove two theorems.
\\
Theorem~\ref{theorem:crn}
(Section~\ref{sec:CrossingNumber}):
Let $\crn(K)$ denotes the crossing number of a knotoid $K$ then
$\|\CHPlus(K)\| +\|\CHMinus(K)\| \leq [\frac{(\crn(K))^2}4]$.
The inequality above is an analogue of the known lower bound
$|C(k)| \leq [\frac{(\crn(k))^2}8]$ where $C(k)$ denotes the Casson invariant of a knot $k$
(see~\cite{PolyakViro2001}).
It is necessary to emphasize that both these estimates are sharp
in the following sense: there exists infinite families of knotoids and knots, respectively,
for which corresponding inequality becomes equality.
\\
Theorem~\ref{theorem:Knot-type}
(Section~\ref{sec:InS^2}):
if $\CHPlus(K) \not=\CHMinus(K)$ for a knotoid $K$
then the knotoid $K$ is proper (or pure with respect to  Turaev's terminology).

Note one more feature of the Casson-type invariant of a knotoid.
It is well-known (see, for example,~\cite{PolyakViro2001})
that the Casson knot invariant can be defined
using two different combinations of over/under-passes
(we call them the upper skew pairs and the lower skew pairs).
In the classical case two corresponding definitions give two  coinciding invariants,
while in the case of knotoid we obtain two invariants $\CPlus$ and $\CMinus$,
and in general they do not coincide
(for corresponding examples  see Sections~\ref{sec:Example}
and~\ref{sec:Table}).
If $\CPlus(K) \not=\CMinus(K)$ for a knotoid $K$
then $K$ is proper knotoid.
This simple  condition is weaker than analogous conditions via $\CHPlusMinus$,
but it is also quite  effective
(see Remark~4 in Section~\ref{sec:Table}).

The paper is organize as follows.
In Section~\ref{sec:Preliminaries}
we recall definitions we need.
In Section~\ref{sec:SkewPairs}
we consider a direct integer-valued analogue of the Casson invariant,
in particular, we prove an  analogue of the well-known skein relation for the Casson knot invariant.
Section~\ref{sec:HomologicalExtention} is central in the paper.
Here we define a homological extension of Casson-type invariant
and discuss some its applications.
In Section~\ref{sec:Example}
we give three examples of computation  of our invariants.
Section~\ref{sec:Table}
contains a table in which we write values of our invariants of all $31$ proper knotoids
having diagram with at most $5$ crossings.
Finally we give a few remarks concerning the table.

\section{Preliminaries}
\label{sec:Preliminaries}

Throughout $\Sigma$ denotes an oriented surface
with (maybe) non-empty boundary.

\subsection{Knotoids}
\label{sec:Knotoid}

A \emph{knotoid diagram} $D$ in the surface $\Sigma$ 
  is a smooth generic immersion of the (closed) segment $[0,1]$ into $\Sigma$
whose only singularities are transversal
double points endowed with standard over/under-crossing data.
By abuse of the language, we will use the same notation both for the immersion and for its image.
The images of $0$ and~$1$ under
this    immersion are called the    \emph{beginning} and the    \emph{end} of $D$, respectively.
   These two points are   distinct from each  other and from the double points; they are 
  called the \emph{endpoints} of~$D$.
Below we think that all knotoid diagrams are oriented from the beginning to the end.
If $\partial \Sigma \not=\emptyset$ then the endpoints of a knotoid diagram
can lie in the boundary of the surface
while all other points of the diagram lie in the interior of $\Sigma$.
The double points  of $D$ are    called the \emph{crossings} of $D$.
The set of crossings of $D$ is denoted by $\cross{D}$.
  
Knotoid diagrams $D_1,D_2 \subset \Sigma$ are \emph{(ambient) isotopic}
if there is an isotopy of $\Sigma$ in itself transforming $D_1$ into $D_2$.
In particular, an isotopy of a knotoid diagram   may displace the endpoints
but if an endpoint lies on $\partial \Sigma$ then the isotopy may move it along corresponding boundary component only.

We define  three  \emph{Reidemeister moves} $\Omega_1, \Omega_2, \Omega_3$ on   knotoid
diagrams.  The move $\Omega_i$ 
on a    knotoid diagram $D$ preserves~$D$ outside a closed $2$-disk disjoint
   from the endpoints and modifies $D$ within this
   disk   as the standard  $i$-th Reidemeister move, for $i=1,2,3$  (pushing a branch of
$D$ over/under the endpoints is not
   allowed).
   
A \emph{knotoid} is defined to be an equivalence class of knotoid diagrams
under the Reidemeister moves $\Omega_1,\Omega_2,\Omega_3$ and ambient isotopies.

\subsection{The sign of a crossing}
\label{sec:Sign}

Recall the standard notion of the sign of a crossing.
Given a knotoid diagram $D \Sigma$ and $x \in \cross{D}$,
denote by $\overrightarrow{b_1},\overrightarrow{b_2}$
the positive tangent vectors of over-crossing and under-crossing branches of $D$ at the crossing $x$, respectively.
Then $\sign(x) =1$ if the pair $(\overrightarrow{b_1},\overrightarrow{b_2})$
forms a positive basis in the tangent space of $\Sigma$ at the point $x$ with respect to the orientation of $\Sigma$.
Otherwise $\sign(x)=-1$.

\subsection{The Gauss code}
\label{sec:GaussCode}

The \emph{Gauss code} of a knotoid diagram
is defined by analogy with the Gauss code of a diagram of a classical knot.
The code consists of letters from a finite alphabet equipped with the signs plus or minus.
The alphabet is in a bijective correspondence with the set of crossings of the diagram.
The sign``$+$'' (resp. ``$-$'') means the overpassing
(resp. underpassing) through corresponding crossing.
Such a pairs (sign, letter) is called
an \emph{items} of the code.
Unlike the classical case in which the Gauss code is defined up to cyclic permutation
the Gauss code of a knotoid diagram is uniquely determined.
To obtain Gauss code of a diagram
we label all crossings with corresponding letters
and then go along the diagram from the beginning to the end.
When we meet a crossing we write its label with the sign plus or minus
depending on which branch (the upper or the lower) we walk along.
In Section~\ref{sec:Example}
we give three examples of Gauss code of knotoid diagrams.

In the case of knotoid
we can say that an  item is placed in the Gauss code before another item
and write $item_1 < item_2$.
It is necessary to emphasize that the expression means that $item_1$ is placed in the code before $item_2$
but not necessary just before,
i.e., some other items can be placed between the two.

\subsection{The arrow diagrams}
\label{sec:ArrowDiagram}

In addition to the language of the Gauss code
we will use another well-known language --- the Gauss arrow diagrams.
In the case of knotoids the definition of an arrow diagram is completely analogous to that in the classical case
with the only obvious difference --- endpoints of arrows lie not on the circle
but on the oriented segment.
As usual the arrow representing a crossing
is oriented from the over-crossing to corresponding under-crossing.

\section{Direct analog of Casson knot invariant}
\label{sec:SkewPairs}

In this section we consider a direct analogue of the Casson knot invariant.
The approach we use follows to that in~\cite{PolyakViro2001}.

Given a knotoid diagram $D$,
an ordered pair $(x,y)$, $x,y \in \cross{D},x \not=y,$   is called
the \emph{upper skew pair} (resp. the \emph{lower skew pair})
if $\Plus{x} < \Minus{y} < \Minus{x} < \Plus{y}$
(resp. $\Minus{x} < \Plus{y} < \Plus{x} < \Minus{y}$).
In Fig.~\ref{fig:1a} we show a fragment of a diagram which gives
an important particular case of a skew pair.
In Figs.~\ref{fig:1b} and~\ref{fig:1c} we draw subdiagrams of an arrow diagram
characterizing the upper and the lower skew pairs, respectively.

\begin{figure}[h!]
\centering 
       \begin{subfigure}[b]{0.3\textwidth}
                \centering
                \psfig{figure=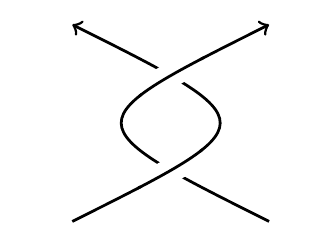, width=\textwidth}
                \caption{}
                \label{fig:1a}
        \end{subfigure}%
       \begin{subfigure}[b]{0.3\textwidth}
                \centering
                \psfig{figure=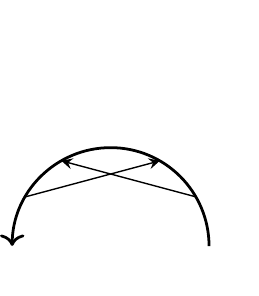, width=\textwidth}
                \caption{}
                \label{fig:1b}
        \end{subfigure}%
       \begin{subfigure}[b]{0.3\textwidth}
                \centering
                \psfig{figure=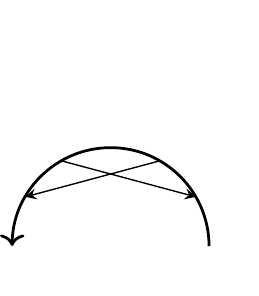, width=\textwidth}
\caption{}
                \label{fig:1c}
        \end{subfigure}%
\caption{}
\label{fig:1}
\end{figure}

Denote by $\PairPlus{D}$ (resp. $\PairMinus{D}$) the subset of $\cross{D}\times\cross{D}$
consisting of all the upper (resp. the lower) skew pairs of crossings.

Let $(x,y)$ is a skew (either upper or lower) pair of crossings.
The \emph{sign of the skew pair} is defined to be
\begin{equation*} \label{eq:SignOfPair}
\sign(x,y) =\sign(x) \sign(y).
\end{equation*}

Following value is a complete analogue  of the Casson knot invariant in the case of knotoid.
\begin{equation} \label{eq:CPlusMinus}
\CPlusMinus(D) =\sum_{(x,y) \in \PairPlusMinus{D}} \sign(x,y)
\in \mathbb{Z}.
\end{equation}

In the classical case $\CPlus =\CMinus$
(see~\cite[Corollary 1.C]{PolyakViro2001})
while in the case of knotoids these value can be different
(see examples in Sections~\ref{sec:K2_1} and~\ref{sec:K4_6} below).

\begin{theorem} \label{theorem:CPlusMinus}
If $D_1,D_2$ are two diagrams of the same knotoid $K$
then
$$\CPlusMinus(D_1) =\CPlusMinus(D_2).$$
\end{theorem}
Therefore, one can say about $\CPlusMinus(K)$
and the values are an invariants of a knotoid $K$.

Theorem~\ref{theorem:CPlusMinus}
is a consequence of Theorem~\ref{theorem:CHPlusMinus}
in Section~\ref{sec:HomologicalExtention} below.

Following properties of $\CPlusMinus$ are straight forward.

\begin{lemma}
Let $D \subset \Sigma$ be a knotoid diagram in the surface $\Sigma$.

{\bf 1.}
If $F: \Sigma \to \Sigma$ is (maybe) orientation reversing homeomorphism of the surface
and $D_1 =F(D)$ then
$\CPlusMinus(D_1) =\CPlusMinus(D)$.

{\bf 2.}
If the diagram $D_2$ is obtained from $D$ by simultaneous switching of all its crossings 
(i.e., all over-crossings in $D$ are replaced with under-crossings and vice versa)
then
$\CPlusMinus(D_2) =\CMinusPlus(D)$.

{\bf 3.}
If the diagram $D_3$ is obtained from $D$ by reversing of the orientation of the diagram then
$\CPlusMinus(D_3) =\CPlusMinus(D)$.

{\bf 4.}
If $K_1,K_2$ are knotoids in surfaces $\Sigma_1,\Sigma_2$, respectively,
then
$$\CPlusMinus(K_1 \cdot K_2) =\CPlusMinus(K_1) +\CPlusMinus(K_2)$$
where $K_1 \cdot K_2$ denotes the product of $K_1$ and $K_2$ in the sense of the semigroup of knotoids
(see~\cite{TuraevKnotoids}).
\end{lemma}

Consider three diagrams $D_0,D_1,D_2$
which form a standard Conway triple,
i.e., the diagrams coincide outside a small neighborhood  of a crossing $x$
while inside the neighborhood they look  like it is shown in
Fig.~\ref{fig:2}.

\begin{figure}[h!]
\centerline{\psfig{figure=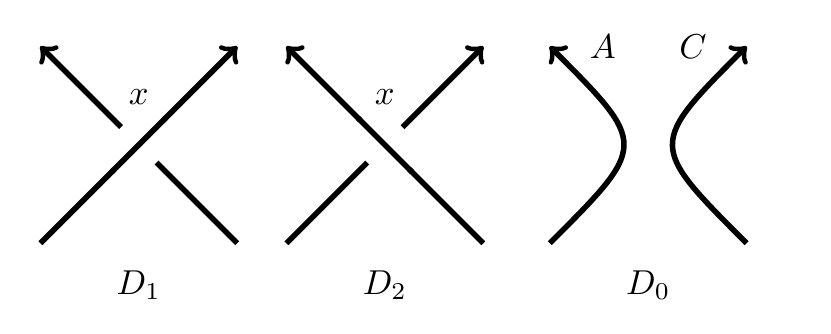}} 
\caption{A Conway triple}
\label{fig:2}
\end{figure}

Here diagrams $D_1,D_2$ are knotoid diagrams
while $D_0$ is a diagram of a multi-knotoid,
i.e., an immersion of disconnected $1$-manifold ---
the disjoint union of the segment and the circle.
Denote parts of $D_0$ by $A$ and $C$
where $A$ denotes the arc and $C$ denotes the circle.
Denote by $s_1,s_2$ the signs of $x$ in diagrams $D_1,D_2$, respectively.

Let the notation is chosen such that $\Plus{x} < \Minus{x}$ in $D_1$.
It is necessary to emphasize that in our notation $s_1$ may be equal to $-1$.

Denote by
\[
\begin{split}
\lkPlus(D_0) =s_1 \sum_{{\underset{\text{$A$ is over $C$}}{y \in A \cap C}}}\sign(y),\\
\lkMinus(D_0) =s_1 \sum_{{\underset{\text{$A$ is under $C$}}{y \in A \cap C}}}\sign(y).
\end{split}
\]

Theorem~\ref{theorem:SkeinRelation}
below states an analogue of known skein relation for the Casson knot invariant:
$C(D_1) -C(D_2) =s_1 \lk(D_0)$
(here $C$ denotes the Casson knot invariant and $\lk$ denotes the linking number).

\begin{theorem} \label{theorem:SkeinRelation}
If $D_0,D_1,D_2$ are a Conway triple as above
then
\begin{equation*}
\CPlus(D_1) -\CPlus(D_2) =s_1 \lkPlus(D_0), \quad 
\CMinus(D_1) -\CMinus(D_2) =s_1 \lkMinus(D_0).
\end{equation*}
\end{theorem}

Note that unlike the classical situation in the case of knotoid the values $\lkPlus(D_0)$ and $\lkMinus(D_0)$
are not necessary equal. It may be so, for example, if $c$ is not homology trivial
or if an endpoint of the diagram lies inside the circle.
 
\begin{proof}[of Theorem~\ref{theorem:SkeinRelation}]
Note if a pair $(u,v),x \not=u,x \not=v,$
is skew then the pair is skew both in $D_1$ and $D_2$.
Thus the terms corresponding to the pair cancel out in the expression $\CPlusMinus(D_1) -\CPlusMinus(D_2)$.
Hence it is sufficient to consider pairs in which the crossing $x$ is involved.

Assume  $y$ is a self-crossing of of the arc $A$.
Then in $D_1$ we have
$\PlusMinus{y} < \Plus{x} < \Minus{x} < \MinusPlus{y}$
hence the pair $(y,x)$ is not skew both in $D_1$ and $D_2$.
If $y$ is a self-crossing of the circle $C$ then
$\Plus{x} <\PlusMinus{y} < \MinusPlus{y} < \Minus{x}$
hence the pair $(x,y)$ is not skew both in $D_1$ and $D_2$.

Assume $y \not=x$ is a crossing in which $A$ and $C$ cross each other.
It means that one of entries of $y$ places between $\Plus{x}$ and $\Minus{x}$
while the other one places either before or after the two,
i.e., there are $4$ possible situations.

{\bf 1.}  Let
$\Plus{y} <\Plus{x} <\Minus{y} <\Minus{x}$
in the diagram $D_1$. Then in $D_2$ we have
$\Plus{y} <\Minus{x} <\Minus{y} <\Plus{x}$.
Hence the pair $(y,x)$ is not skew in $D_1$ and
in $D_2$ it is upper skew pair such that $\sign(y,x)=s_2 \sign(y) =-s_1 \sign(y)$ .

{\bf 2.} Let
$\Plus{x} < \Minus{y} <\Minus{x} <\Plus{y}$ in $D_1$.
Then in $D_2$ we have
$\Minus{x} <\Minus{y} <\Plus{x} <\Plus{y}$.
Hence in $D_1$ the pair $(x,y)$ is the upper skew pair such that $\sign(x,y) =s_1 \sign(y)$
while in $D_2$ the pair is not skew.

In both cases above $y$ is such an intersection point of $A$ and $C$
in which $A$ is on top.
In two cases below $A$ is under $C$.

{\bf 3.} Let
$\Minus{y} <\Plus{x} <\Plus{y} <\Minus{x}$ in $D_1$.
then in $D_2$ we have
$\Minus{y} <\Minus{x} <\Plus{y} <\Plus{x}$.
Hence in $D_1$ the pair $(y,x)$ is the lower skew pair such that $\sign(y,x) =s_1 \sign(y)$
while in $D_2$ the pair is not skew.

{\bf 4.} Let
$\Plus{x} <\Plus{y} <\Minus{x} <\Minus{y}$ in $D_1$.
then in $D_2$ we have
$\Minus{x} <\Plus{y} <\Plus{x} <\Minus{x}$.
Hence the pair $(x,y)$ is not skew in $D_1$
while in $D_2$ it is the lower skew pair such that $\sign(x,y) =s_2 \sign(y) =-s_1 \sign(y)$.

Therefore, upper skew pairs exist in the cases $1$ and $2$ in which $A$ is over $C$
while lower skew pairs exist in the cases $3$ and $4$ in which $A$ is under $C$.
The sign of the skew pairs existing in $D_1$ is equal to $s_1$
while the sign of such a pairs in $D_2$ is equal to $s_2=-s_1$.
These observations mean that the proving equalities hold.
\end{proof}

{\bf Remark.}
The Casson knot invariant can be extended to virtual knotoids also.
The notion of virtual knotoid is introduced in~\cite{KauffmanInvariants}.
A virtual knotoid is an equivalence class of virtual knotoid diagrams.
A virtual knotoid diagram is defined analogously to usual knotoid diagrams in $S^2$
with the only difference --- crossing in the diagram are either classical (that is usual crossings)
or virtual which are not provided with over/under-crossing data.
Such a diagrams are regarded up to isotopies and extended set of Reidemeister moves
which coincides with that in virtual knot theory.
All definitions above can be extended to virtual knotoids without any changes.
Values $\CPlusMinus$ are invariants of virtual knotoids
because additional (virtual) Reidemeister moves do not effect on skew pairs.

\section{A homological extension of $\CPlusMinus$}
\label{sec:HomologicalExtention}

In this section we use the notion of a skew pair
to define another invariant of a knotoid.
Below we additionally associate with a skew pair    a subgroup of $H_1(\Sigma)$ determined by  the pair.

\subsection{The definition of $\CHPlusMinus$}
\label{sec:CHPlusMinus}

Given a knotoid diagram $D$ and $x \in \cross{D}$,
the crossing divides the diagram into two parts:
the loop which starts and ends at $x$  (we denote it by $l(x)$)
and the arc (maybe with self-crossings)
which starts at the beginning of $D$, passes through $x$ and then goes to the end of $D$.
These two part are that appear as a result of orientation agree smoothing of $x$.

Let $(x,y)$ be a skew pair.
We associate with the pair 
an ordered pair
$$(L_1(x,y)=l(x),L_2(x,y)=l(y))$$
of loops in $\Sigma$.

The subgroup of $H_1(\Sigma)$ which we associate with the pair $(x,y)$ is defined as follows
\begin{equation*} \label{eq:H(x,y)}
H(x,y) =\Angle{[L_1(x,y)], [L_2(x,y)]}
\end{equation*}
here $[\cdot]$ denotes the homology class of the corresponding loop
and $\Angle{\ldots}$ denotes the subgroup generated by specified elements.

Denote by $\ModM =\ModM(\Sigma)$ the free $\mathbb{Z}$-module freely generated
by elements of $L(H_1(\Sigma))$, 
where $L(H_1(\Sigma))$ denotes the lattice of subgroups of the group $H_1(\Sigma)$.

Set
\begin{equation} \label{eq:CHPlusMinus}
\CHPlusMinus(D) =\sum_{(x,y) \in \PairPlusMinus{D}}
\sign(x,y) H(x,y) \in \ModM.
\end{equation}

Note (cf.~\eqref{eq:CPlusMinus})
that $\CHPlusMinus$ turns into $\CPlusMinus$
if we ignore homological factors.

\begin{theorem} \label{theorem:CHPlusMinus}
If $D_1,D_2$ are diagrams of the same knotoid $K$
then 
$$\CHPlusMinus(D_1) =\CHPlusMinus(D_2) \in \ModM.$$
\end{theorem}
Hence one can say about $\CHPlusMinus(K) \in \ModM$
and the value is an invariant of $K$ up to automorphism of the group $H_1(\Sigma)$.

\begin{proof}
To prove the theorem we need to check that the values $\CHPlusMinus$ are preserved under all oriented Reidemeister moves.
In~\cite{Polyak2010} Polyak proved that 
there exists a generating set of oriented Reidemeister moves consisting of $4$ moves only
(see Fig.~\ref{fig:3_4_5}):
$2$ versions of the first Reidemeister move, $1$ version of the second one
and $1$ version of the third one.
Polyak's theorem is originally proved for the case of diagrams of classical knots
but all considerations within the proof are performed inside a disk,
therefore, the theorem  without any changes can be extended to the case of knotoid diagrams.
Hence it is sufficient to check that
$\CHPlusMinus$ is preserved under these generating moves only.

\begin{figure}[h!]
\centering 
\begin{subfigure}[b] {0.387\textwidth}
\centering
\psfig{figure=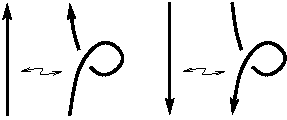, width=\textwidth}
\caption{}
\label{fig:3}                                        
\end{subfigure}%
\begin{subfigure}[b] {0.213\textwidth}
\centering
\psfig{figure=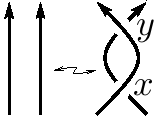, width=\textwidth}
\caption{}
\label{fig:4}
\end{subfigure}%
\begin{subfigure}[b] {0.4\textwidth}
\centering
\psfig{figure=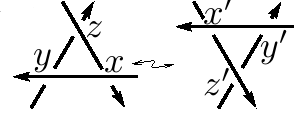, width=\textwidth}
\caption{}
\label{fig:5}
\end{subfigure}%
\caption{The generating set of oriented Reidemeister moves}
\label{fig:3_4_5}
\end{figure}

{\bf The first Reidemeister move.} (see Fig.~\ref{fig:3})
The crossing which appears/disappears as a result of the move
can not be involved in a skew (both upper and lower) pair
because passes through the crossing are placed one just after another.
Hence the move does not effect on $\CHPlusMinus$.

{\bf The second Reidemeister move.} (see Fig.~\ref{fig:4})
Denote by $x,y$ the two crossings which are vertices of the bigon appearing/disappearing as a result of the move
and let the notation is chosen such that
$\Plus{x} < \Plus{y}$ and $\Minus{x} < \Minus{y}$.
The pair $(x,y)$ is not skew
because $\Plus{y}$ is placed just after $\Plus{x}$.
The pair $(y,x)$is not skew also
because going along the diagram we meet $x$ before $y$.
Assume there is a crossing $z$ such that the pair $(x,z)$ is the upper skew pair.
Since $\Plus{y}$ and $\Minus{y}$ are placed just after $\Plus{x}$ and $\Minus{x}$, respectively,
the pair $(y,z)$ is the upper skew pair also.
The union of arcs $[\Plus{x},\Plus{y}]$ and $[\Minus{x},\Minus{y}]$ bounds a disk
hence the loop $l_j(x,z)$ is homologous to the loops $l_j(y,z)$ for $j=1,2,$
thus $H(x,z) =H(y,z) \in \ModM$.
Therefore, terms corresponding to $(x,z)$ and $(y,z)$ in the sum~\eqref{eq:CHPlusMinus} cancel out
because $\sign(x) =-\sign(y)$.
All other cases (the pair $(x,z)$ is lower or the pair $(z,x)$ is skew)
are completely analogous to the case above,
hence the move does not effect on $\CHPlusMinus$.

{\bf The third Reidemeister move.} (see Fig.~\ref{fig:5})
Denote by $R$ the disk inside which the move is performed
and by $x,y,z$, $x',y',z'=z$ the crossings involved in the move.
Let the notation is chosen such that
$$\sign(x)=\sign(x')=1,\quad \sign(y)=\sign(y')=-1, \quad \sign(z)=\sign(z')=1.$$
Assume there is a crossing $v \in \cross{D}$  lying outside the disk $R$ such that the pair $(x,v)$ is a skew pair.
Then the pair $(x',v)$ is skew also
and $\sign(x,v) =\sign(x',v)$.
Since outside $R$ the diagram does not change
the loop $l_j(x,v)$ is homologous to the loop $l_j(x',v)$ for $j=1,2$.
Hence corresponding terms in the sum~\eqref{eq:CHPlusMinus} coincide before and after the move.
The same holds for all crossings involved in the transformation
equally for upper and lower pairs.

Therefore, it remains to consider the case when both crossings in the pair lie inside $R$.
Observe that there are two different (up to cyclic permutation) orders
in which one can traverse through $R$ on the way from the beginning to the end of the diagram.
First we consider the left fragment in Fig.~\ref{fig:5}.

{\bf (I)}
$\ldots \Plus{x},\Plus{y}, \ldots \Minus{y},\Minus{z}, \ldots \Plus{z},\Minus{x} \ldots$
(see Fig.~\ref{fig:6})
Here and below ``$\ldots$'' stands for parts of the diagram placed outside the disk $R$.
In this case no two of these three crossings form a skew pair irrespective of the arc
along which one comes into the fragment for the first time.
That is because no two arrows (of three arrows under consideration) cross each other.

\begin{figure}[h!]
\centering 
       \begin{subfigure}[b]{0.3\textwidth}
                \centering
                \psfig{figure=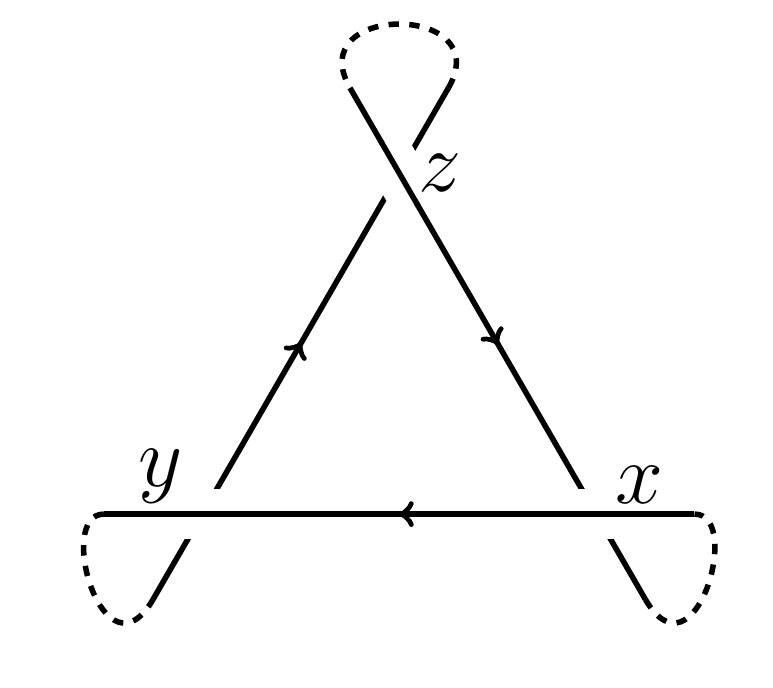, width=\textwidth}
                                        \end{subfigure}%
       \begin{subfigure}[b]{0.3\textwidth}
                \centering
                \psfig{figure=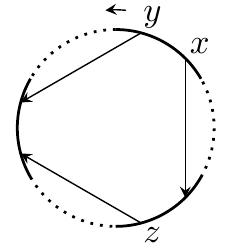, width=\textwidth}
                        \end{subfigure}%
                        \caption{The case (I)}
\label{fig:6}
\end{figure}

{\bf (II)}
$\ldots \Plus{x},\Plus{y},\ldots,\Plus{z},\Minus{x},\ldots,\Minus{y},\Minus{z},\ldots$
In Fig.~\ref{fig:7_0}
corresponding fragment of the arrow diagram  is depicted.
Now the arrows are pairwise intersecting
and the existence/non-existence of a skew pairs depends on 
from which side one comes into $R$ for the first time.

\begin{figure}[h!]
\centering 
\begin{subfigure}[b]{0.22\textwidth}
                \centering
                \psfig{figure=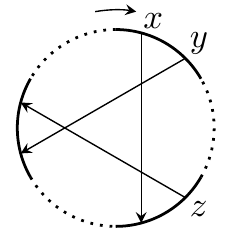, width=\textwidth}
                \caption{}
                \label{fig:7_0}
        \end{subfigure}%
              \begin{subfigure}[b]{0.22\textwidth}
                \centering
                \psfig{figure=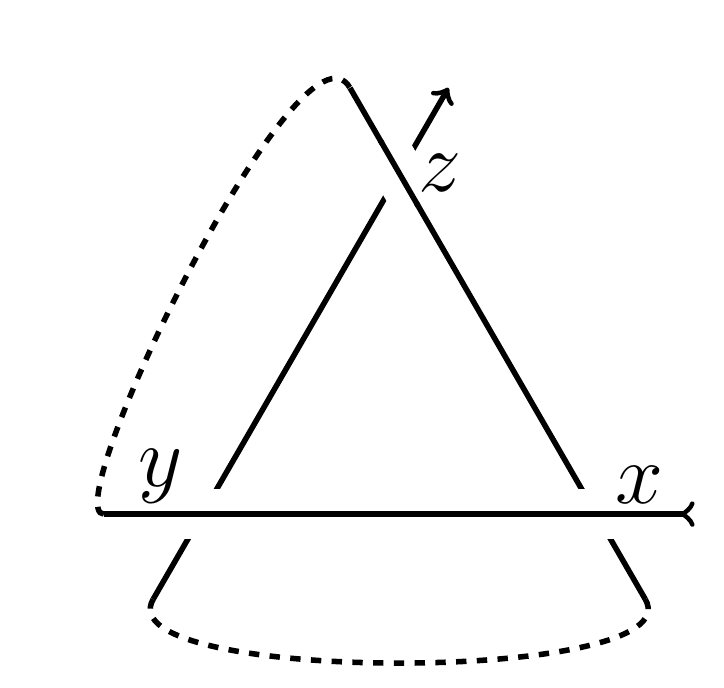, width=\textwidth}
                \caption{}
                \label{fig:7a}
        \end{subfigure}%
       \begin{subfigure}[b]{0.22\textwidth}
                \centering
                \psfig{figure=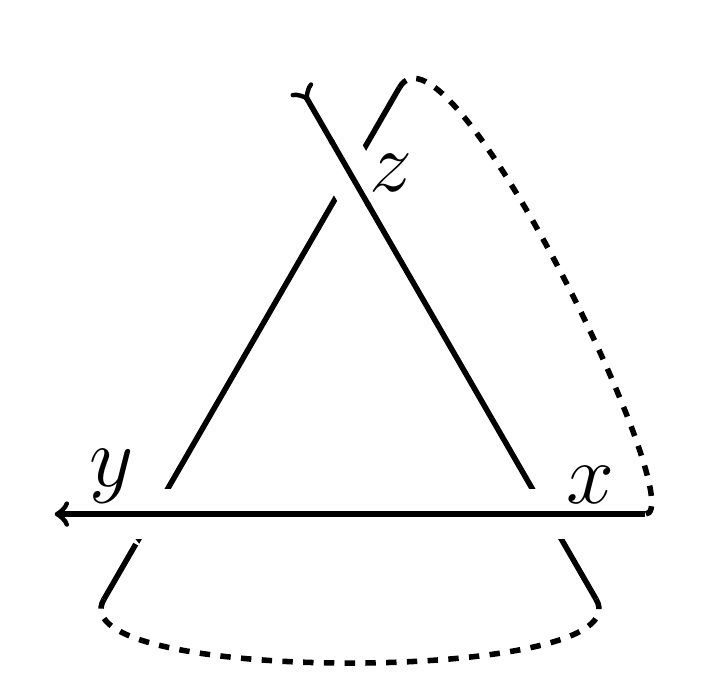, width=\textwidth}
                \caption{}
                \label{fig:7b}
        \end{subfigure}%
       \begin{subfigure}[b]{0.22\textwidth}
                \centering
                \psfig{figure=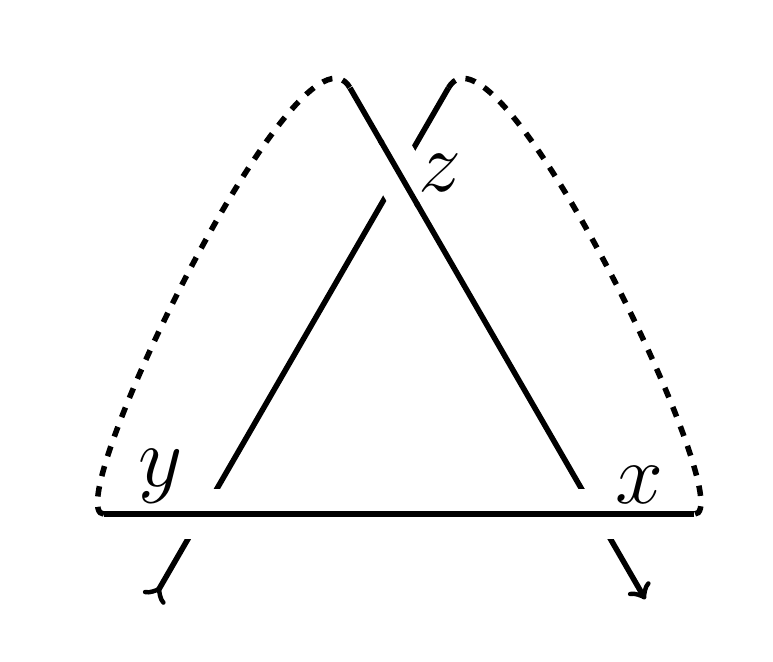, width=\textwidth}
                \caption{}
                \label{fig:7c}
        \end{subfigure}%
\caption{\small{
(b) No skew pair,
(c) $(z,x), (z,y) \in \PairPlus{D}$,
(d) $(y,x), (z,x) \in \PairMinus{D}$.
}}
\label{fig:7}
\end{figure}

If one comes into the fragment for the first time  along the arc $\Plus{x},\Plus{y}$
(see Fig.~\ref{fig:7a})
then no skew pair exists.

In the case
$\ldots,\Plus{z},\Minus{x},\ldots,\Minus{y},\Minus{z},\ldots,\Plus{x},\Plus{y},\ldots$
(see Fig.~\ref{fig:7b})
there are two upper skew pairs: $(z,x)$ and $(z,y)$
having the opposite signs: $\sign(z,x)=1,\sign(z,y)=-1$.
Check that the subgroups $H(z,x) =H(z,y)$ coincide.
The loops $L_1(z,x)$ and $L_1(z,y)$ coincide with $l(z)$
(see Fig.~\ref{fig:20_1}).
The loops $L_2(z,y)=l(y)$(see Fig.~\ref{fig:20_2}) 
and $L_2(z,x) =l(x)$ (see Fig.~\ref{fig:20_3})
do not coincide
but in $H_1(\Sigma)$ we have
$[L_2(z,x)] =[l(x)] =[l(z)] +[l(y)] =[L_1(z,y)] +[L_2(z,y)]$,
hence the subgroup in question coincide.
Therefore, the terms in the sum~\eqref{eq:CHPlusMinus}
corresponding to the pairs $(z,x)$ and $(z,y)$ cancel out.

\begin{figure}[h!]
\centering 
\begin{subfigure}[b]{0.3\textwidth}
                \centering
                \psfig{figure=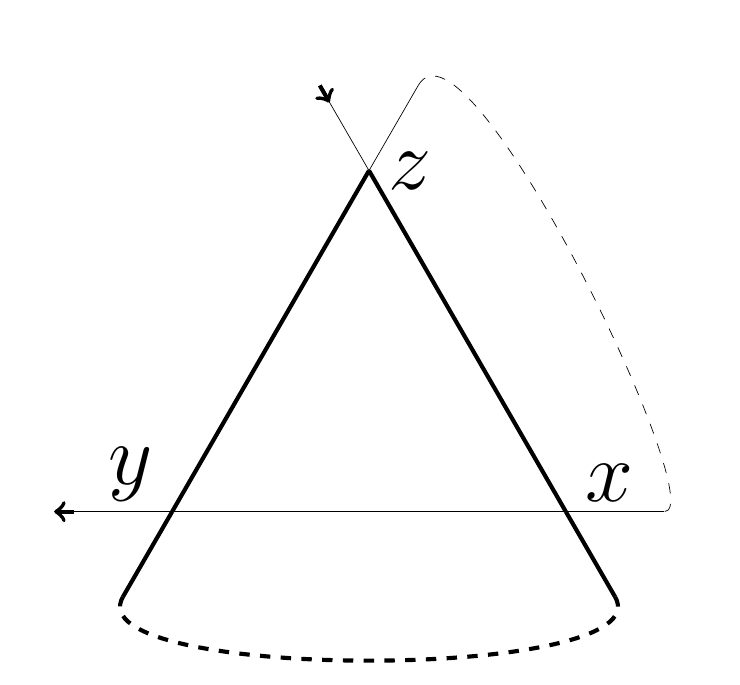, width=\textwidth}
                \caption{$l(z)$}
                \label{fig:20_1}
        \end{subfigure}%
              \begin{subfigure}[b]{0.3\textwidth}
                \centering
                \psfig{figure=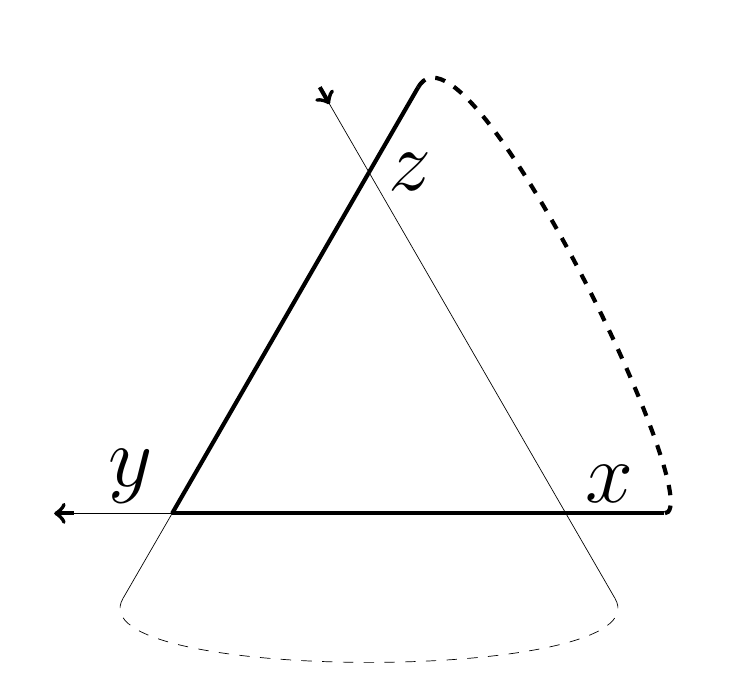, width=\textwidth}
                \caption{$l(y)$}
                \label{fig:20_2}
        \end{subfigure}%
       \begin{subfigure}[b]{0.3\textwidth}
                \centering
                \psfig{figure=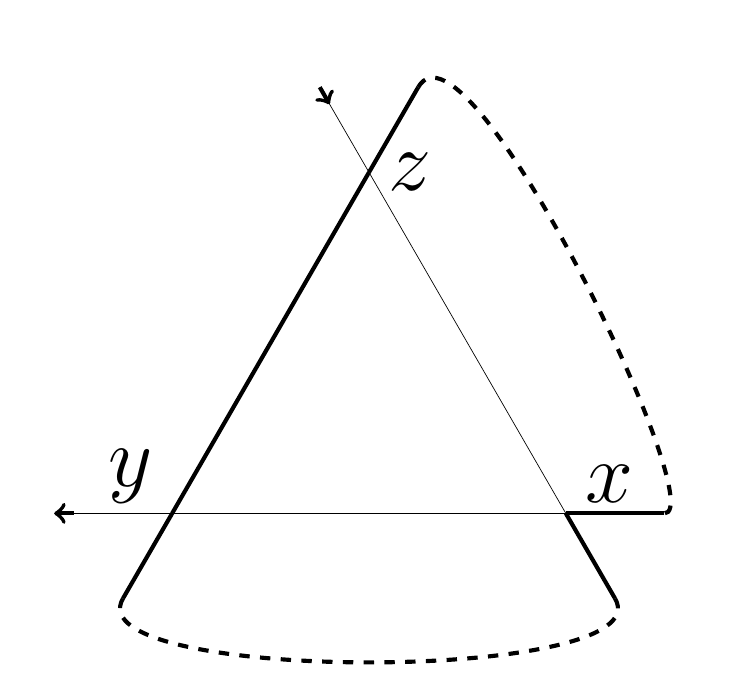, width=\textwidth}
                \caption{$l(x)$}
                \label{fig:20_3}
        \end{subfigure}%
                                                                               \caption{$[l(x)] =[l(y)] +[l(z)]$}
\label{fig:20}
\end{figure}

In the case
$\ldots,\Minus{y},\Minus{z},\ldots,\Plus{x},\Plus{y},\ldots,\Plus{z},\Minus{x},\ldots$
(see Fig.~\ref{fig:7c})
there are two lower skew pairs $(y,x)$ and $(z,x)$ again having the opposite signs:
$\sign(y,x)=-1,\sign(z,x)=1$.
(Now we do not draw loops in question
because this situation and two situations below is like to the one above.)
The loops $l_2(y,x)$ and $l_2(z,x)$ coincide with $l(x)$
and $[L_1(z,x)] =[l(z)] =[l(y)] +[l(x)] =[L_1(y,x)] +[L_2(y,x)] \in H_1(\Sigma)$.
Hence $H(y,x) =H(z,x) \in \ModM$
thus corresponding terms in the sum~\eqref{eq:CHPlusMinus}
cancel out.

Now we consider the right fragment in Fig.~\ref{fig:5}.

{\bf (III)}
(see Fig.~\ref{fig:8}).

\begin{figure}[h!]
\centering 
\begin{subfigure}[t]{0.22\textwidth}
                \centering
                \psfig{figure=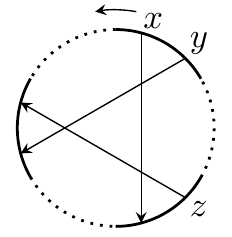, width=\textwidth}
                \caption{}
                \label{fig:8_0}
        \end{subfigure}%
              \begin{subfigure}[t]{0.22\textwidth}
                \centering
                \psfig{figure=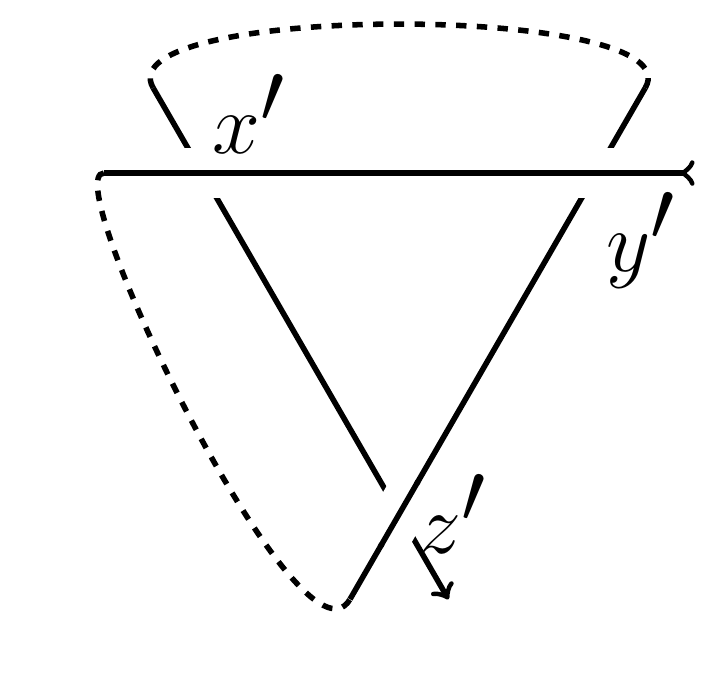, width=\textwidth}
                \caption{}
                \label{fig:8a}
        \end{subfigure}%
       \begin{subfigure}[t]{0.22\textwidth}
                \centering
                \psfig{figure=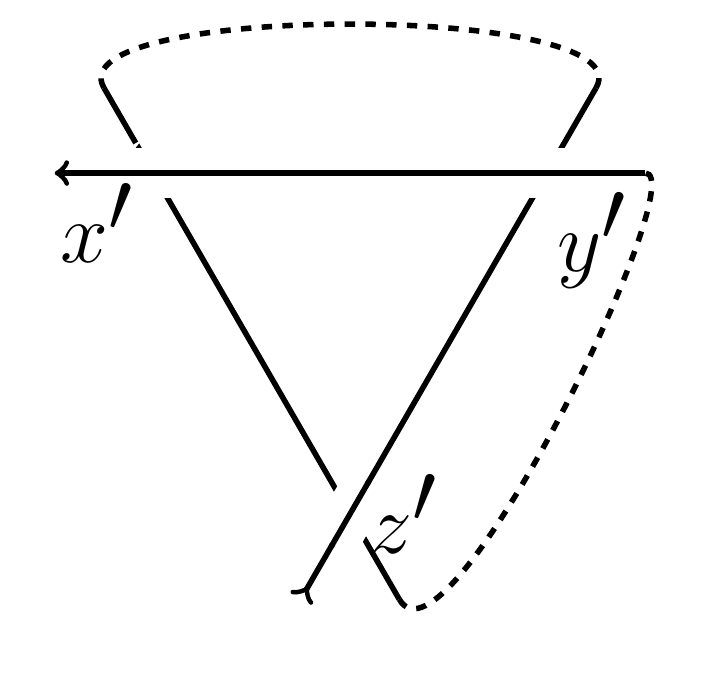, width=\textwidth}
                \caption{}
                \label{fig:8b}
        \end{subfigure}%
       \begin{subfigure}[t]{0.22\textwidth}
                \centering
                \psfig{figure=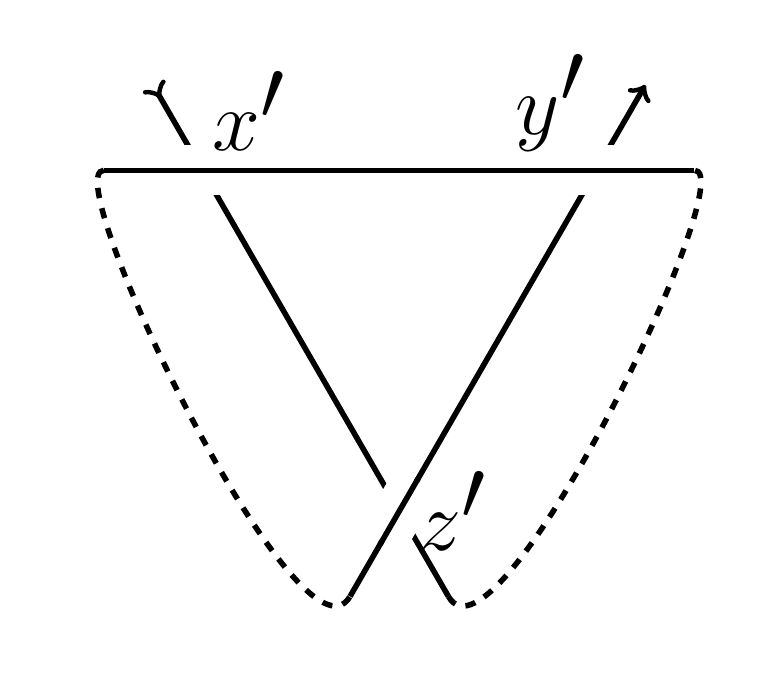, width=\textwidth}
                \caption{}
                \label{fig:8c}
        \end{subfigure}%
\caption{\small{
(b) $(y',z'), (x',z') \in \PairPlus{D}$,
(c) No skew pair,
(d) $(x',y'), (x',z') \in \PairMinus{D}$.
}}
\label{fig:8}
\end{figure}

The existence/non-existence of skew pairs depends on 
from which side one comes into $R$ for the first time.

$\ldots,\Plus{y'},\Plus{x'},\ldots,\Minus{z'},\Minus{y'},\ldots,\Minus{x'},\Plus{z'},\ldots$
(see Fig.~\ref{fig:8a}).
In the specified order there are two upper skew pairs: $(y',z')$ and $(x',z')$
having the opposite signs: $\sign(y',z')=-1,\sign(x',z')=1$.
The loops $l_2(y',z')$ and $l_2(x',z')$ coincide with $l(z')$
and
$[L_1(x',z')] =[l(x')] =[l(y')] +[l(z')] =[L_1(y',z')] +[L_2(y',z')] \in H_1(\Sigma)$.
Hence $H(x',z') =H(y',z') \in \ModM$
and thus corresponding terms in the sum~\eqref{eq:CHPlusMinus}
cancel out.

In the case
$\ldots,\Minus{z'},\Minus{y'},\ldots,\Minus{x'},\Plus{z'},\ldots,\Plus{y'},\Plus{x'},\ldots$
(see Fig.~\ref{fig:8b})
there are no skew pairs.

In the case
$\ldots,\Minus{x'},\Plus{z'},\ldots,\Plus{y'},\Plus{x'},\ldots,\Minus{z'},\Minus{y'},\ldots$
(see Fig.~\ref{fig:8c})
there two lower skew pairs: $(x',z')$ and $(x',y')$
having the opposite signs: $\sign(x',z')=1,\sign(x',y')=-1$.
The loops $L_1(x',z')$ and $L_1(x',y')$ coincide with $l(x')$
and
$[L_2(x',z')] =[l(z')] =[l(x')] +[l(y')] =[L_1(x',y')] +[L_2(x',y')] \in H_1(\Sigma)$,
hence $H(x',z') =H(x',y') \in \ModM$
 thus corresponding terms in the sum~\eqref{eq:CHPlusMinus}
cancel out.

{\bf (IV)}
$\ldots,\Plus{y'},\Plus{x'},\ldots,\Minus{x'},\Plus{z'},\ldots,\Minus{z'},\Minus{y'},\ldots$
(see Fig.~\ref{fig:9}).

\begin{figure}[h!]
\centering 
       \begin{subfigure}[b]{0.3\textwidth}
                \centering
                \psfig{figure=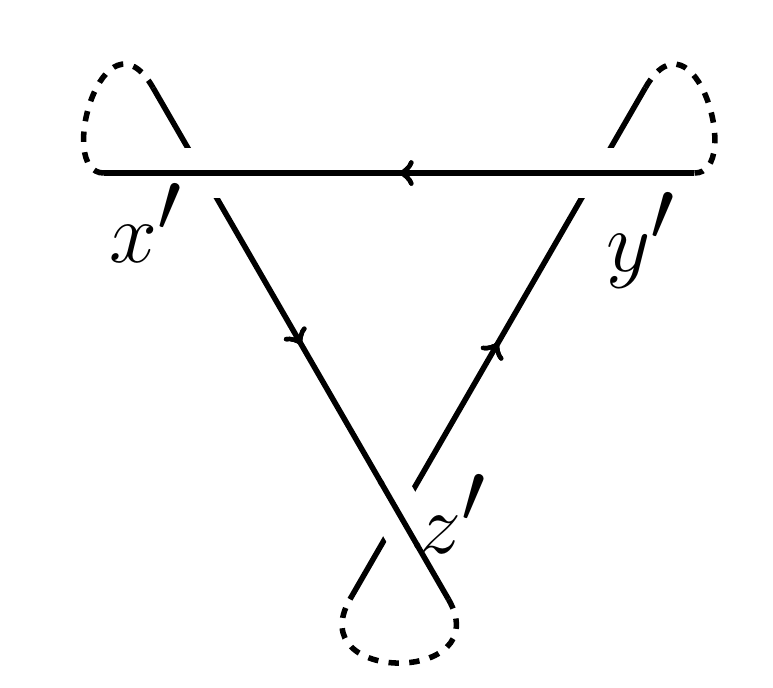, width=\textwidth}
                                        \end{subfigure}%
       \begin{subfigure}[b]{0.3\textwidth}
                \centering
                \psfig{figure=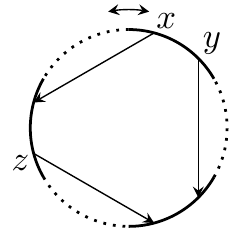, width=\textwidth}
                        \end{subfigure}%
                        \caption{The case (IV)}
\label{fig:9}
\end{figure}

In this case as in the case (I) 
there are no skew pairs irrespective of the side
from which one comes into the fragment for the first time.

Therefore, in all cases above either there are no skew pairs
or two  skew pairs exists
but corresponding terms in the sum~\eqref{eq:CHPlusMinus} cancel out.
Therefore, in all possible situations the third Reidemeister move does not change the values $\CHPlusMinus$.

This completes the proof of Theorem~\ref{theorem:CHPlusMinus}.
\end{proof}

\subsection{A lower bound of the crossing number}
\label{sec:CrossingNumber}

Using $\CHPlusMinus$ we can obtain
a lower bound of the crossing number of a knotoid.
Recall the crossing number of a knotoid $K$ (we denote it by $\crn(K)$)
is defined to be the minimum of crossings over all representative diagrams.
The statement below is an analogue of~\cite[Theorem 1.E]{PolyakViro2001}
which gives a sharp  lower bound of the crossing number of a classical knot by its Casson knot invariant.
But it is necessary to emphasize that
the sharp bound in the case of knotoids is twice the sharp bound in the classical case.

Denote by $\|q\|,q \in \ModM,$ a standard norm of the vector $q$,
namely, if $q =\sum_j q_j H_j,q_j \in \mathbb{Z},H_j \in L(H_1(\Sigma)),$
then
$$\|q\| =\sum_j|q_j|$$
(the sum on the right-hand side is finite by definition of $\ModM$).

\begin{theorem} \label{theorem:crn}
{\bf 1}. For any knotoid $K$
\begin{equation} \label{eq:crn}
\|\CHPlus(K)\| +\|\CHMinus(K)\| \leq [\frac{(\crn(K))^2}4]
\end{equation}
 
{\bf 2}. There exists an infinite family of knotoids in $S^2$
for which the inequality ~\eqref{eq:crn} becomes equality.
\end{theorem}

\begin{proof}
{\bf 1}. 
Given a diagram $D$ of a knotoid $K$
divide the set of its crossings into two  subsets:
\[
\begin{split}
C_{+} &=\{ x \in \cross{D}: \Plus{x} < \Minus{x} \},\\
C_{-} &=\{ x \in \cross{D}: \Minus{x} < \Plus{x} \}.
\end{split}
\]
Denote $n_{\pm}=|C_{\pm}|$
(here $|\cdot|$ denotes the cardinality of the specified set).
Then $n_{+} +n_{-} =n$ where $n =|\cross{D}|$.
Each skew pair of crossings (both upper and lower) consists of two crossings
such that one of them is an element of $C_{+}$
while the other one is an element of $C_{-}$.
Hence the total number of skew pairs
$$|\PairPlus{D}| +|\PairMinus{D}| \leq n_{+} n_{-} =n_{+} (n -n_{+}).$$
Finally note that $\|\CHPlus(K)\| \leq |\PairPlus{D}|$,
$\|\CHMinus(K)\| \leq |\PairMinus{D}|$
and $n_{+} (n -n_{+}) \leq \frac{n^2}4$.
This completes the proof of inequality~\eqref{eq:crn}.

{\bf 2.}
Consider following family of knotoids
$K_j \subset S^2,j=1,2,\ldots$
given by diagrams $D_j$ with Gauss codes:
$$D_j \to \Plus{x_1}, \Minus{x_2},\ldots,\Plus{x_{2j-1}}, \Minus{x_{2j}},\Minus{x_1}, \Plus{x_2},\ldots, \Minus{x_{2j-1}}, \Plus{x_{2j}},$$
and $\sign(x_k)=1,k=1,\ldots,2j$.
The knotoid $2_1$ considered below in Section~\ref{sec:K2_1}
(see Fig.~\ref{fig:K2_1})
is the first knotoid in the family.
The diagram $D_2$ is drawn in Fig.~\ref{fig:K4_1}.
These two figures suggest the way of drawing $D_j$ for any $j \geq 1,$ in $S^2$.
Hence all knotoids in the family are indeed spherical knotoids.

\begin{figure}[h!]
\centering
\psfig{figure=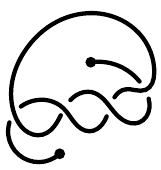}
\caption{The diagram $D_2$}
\label{fig:K4_1}
\end{figure}

Clearly, $n_j=\crn(K_j) \leq 2j$ for any $j \geq 1$.

The set $\PairPlus{D_j}$ consists of following pairs:
\[
\begin{array}{ccccc}
(x_1,x_2), & (x_1,x_4),& \ldots,&(x_1,x_{2j-2}),&(x_1,x_{2j}),\\
&(x_3,x_4),&\ldots,&(x_3,x_{2j-2}),&(x_3,x_{2j}),\\ 
&&\ldots&\ldots&\ldots \\
&&&(x_{2j-3},x_{2j-2}),&(x_{2j-3},x_{2j}),\\
&&&&(x_{2j-1},x_{2j}).
\end{array}
\]
Hence
$$|\PairPlus{D_j}| =j +(j-1) +(j-2) +\ldots +1 =\frac{j(j+1)}2.$$
Similarly, the set $\PairMinus{D_j}$ consists of
\[
\begin{array}{ccccc}
(x_2,x_3), & (x_2,x_5),& \ldots,&(x_2,x_{2j-3}),&(x_2,x_{2j-1}),\\
&(x_4,x_5),&\ldots,&(x_4,x_{2j-3}),&(x_4,x_{2j-1}),\\ 
&&\ldots&\ldots&\ldots \\
&&&(x_{2j-4},x_{2j-3}),&(x_{2j-4},x_{2j-1}),\\
&&&&(x_{2j-2},x_{2j-1}).
\end{array}
\]
Hence
$$|\PairMinus{D_j}| =(j-1) +(j-2) +\ldots +1 =\frac{j(j-1)}2.$$
Thus
$$|\PairPlus{D_j}| +|\PairMinus{D_j}| =\frac{j(j+1)}2 +\frac{j(j-1)}2 =j^2.$$
It remains to note that $j=\frac{n_j}2$ and 
that since all crossings in the diagram are positive then
$|\PairPlus{D_j}| =\| \CHPlus(K)\|$,
$|\PairMinus{D_j}| =\|\CHMinus(K)\|$.
	\end{proof}

{\bf Remarks}

1. 
The family of knotoids used above in the proof of the second part of Theorem~\ref{theorem:crn}
consists of spherical knotoids having even crossing number.
We know examples of knotoids having odd crossing number for which the inequality~\eqref{eq:crn}
becomes equality but the knotoids are not spherical.
So we do not know whether the estimate~\eqref{eq:crn} is sharp
for spherical knotoids having odd crossing number.
The table in Section~\ref{sec:Table} 
suggests the conjecture that in this case the bound should be decresed by $1$
(see Remark~5 in Section~\ref{sec:Table}).

2. The proof of the first part of Theorem~\ref{theorem:crn}
 follows proof of~\cite[Theorem 1.E]{PolyakViro2001}
until the last step
in which the difference between the Casson knot invariant and its analogue for knotoids becomes apparent.
In our case $\CHPlus$ and $\CHMinus$ (similarly $\CPlus$ and $\CMinus$)
are two different invariants
while in the classical  case the sum over the upper skew pairs coincides
with the sum over the lower skew pairs
and thus for classical knot $\CPlus +\CMinus =2C$.
That is because for knotoid we have $\frac{n^2}4$ while for classical knot the bound is $\frac{n^2}8$.

3. Theorem~\ref{theorem:Knot-type} shows that the difference $\CPlus -\CMinus$
keeps an useful information about a knotoid.

\subsection{$\CHPlusMinus$ of knotoids in $S^2$}
\label{sec:InS^2}

The first homology group of the sphere $S^2$ is trivial,
however, $\CHPlusMinus$ can be useful in the case of  classical knotoids also.
To this end we use a standard trick:
we transform a knotoid in $S^2$ into a knotoid in the annulus $A$.
To realize the transformation we take a diagram of the knotoid under consideration
and remove from $S^2$ two small open disks
around the endpoints of the diagram.
Resulting diagram is a knotoid diagram in $A$
whose endpoints lie on distinct connected components of $\partial A$.
It is easy to check that
as a result of such transformation the original equivalence relation on the set of knotoid diagrams in $S^2$
comes into the equivalence relation on the set of knotoid diagrams in $A$.
Since $H_1(A)$ is isomorphic to $\mathbb{Z}$
the basis of the module $\ModM(A)$ is infinite
(recall the basis consists of subgroups of the first homology group of corresponding surface)
hence $\CHPlusMinus$ may be more informative than $\CPlusMinus$.
Examples considered in Section~\ref{sec:Example}
and the table in Section~\ref{sec:Table}
show that it is indeed so.

An element of the module $\ModM(A)$
(and thus values of $\CHPlusMinus$)
in this case can be given by following sums:
$$\sum_{j=0}^{\infty} c_j \Angle{j},$$
where $c_j \in \mathbb{Z}$
and $\Angle{j}$ stands for the subgroup of $\mathbb{Z}$ generated by $j$.
In particular, $\Angle{1}$ stands for $\mathbb{Z}$ and $\Angle{0}$ stands for its trivial subgroup.

Now we prove a simple sufficient conditions for a knotoid in $S^2$
to be a knot-type knotoid.
Recall a knotoid $K \subset S^2$ is called a \emph{knot-type knotoid}
if $K$ admits a diagram $D$
for which there exists a simple arc $s \subset S^2$ 
whose endpoints coincide with the endpoints of $D$
and such that $\Int s \cap D =\emptyset$
where $\Int s$ denotes the interior of $s$.
If a knotoid is not a knot-type then
it is called a \emph{proper knotoid}
(or a \emph{pure knotoid} with respect  to Turaev's terminology in~\cite{TuraevKnotoids}).

\begin{theorem} \label{theorem:Knot-type}
Let $K$ be a spherical knotoid
for which at least one of following conditions holds:

{\bf (i)}
$\CPlus(K) \not=\CMinus(K)$,

{\bf (ii)}
$\CHPlusMinus(K) \not=\CPlusMinus(K) \Angle{0},$	

then the knotoid $K$ is a proper knotoid.
\end{theorem}

\begin{proof}
Assume the contrary, i.e., that $K$ is a knot-type knotoid.
Let $D$ and $s$ be a diagram of $K$ and the arc
from the definition of a knot-type knotoid.
The union $D \cup s$ can be regarded as a classical knot diagram.
Then $\CPlus(D \cup s) =\CMinus(D \cup s)$
and the value do not depend on the choice of a base point.
So we can place the base point into the beginning of the knotoid diagram $D$.
Then a pair of crossings is skew with respect to  the classical knot diagram $D \cup s$
if and only if the pair is skew with respect to the knotoid diagram $D$.
Hence $\CPlusMinus(D) =\CPlusMinus(D \cup s) =\CMinusPlus(D)$
contradicting the condition {\bf (i)}.

To prove {\bf (ii)} assume 
that the value $\CHPlus(D)$ (or $\CHMinus(D)$) contains a term with non-trivial subgroup of $H_1(A)$.
Such a term can appear if there is a crossing $x \in \cross{D}$
for which $l(x)$ is not homology trivial.
But that is impossible because
such a loop should at least once go along the axial line of the annulus
and thus intersect the arc $s$
contradicting the condition $D \cap \Int s =\emptyset$ from the definition of a knot-type knotoid.
Thus the subgroup $H(x,y)$ is trivial for a skew pair $(x,y)$ (if any).
Hence all terms in the sum $\CHPlusMinus(D)$ are given in the form: the sign of corresponding skew pair multiplying to $\Angle{0}$
and the resulting value is equal to $\CPlusMinus(D) \Angle{0}$.
\end{proof}

{\bf Remarks.}

1. The conditions are quite effective
(see Remark~4 in Section~\ref{sec:Table})
but the example considered in Section~\ref{sec:K5_19}
shows that they are not necessary.

2. By analogy with the proof of the first part of Theorem~\ref{theorem:Knot-type}
we can prove that
if a virtual knotoid $K$
is such that $\CPlus(K) \not=\CMinus(K)$
then $K$ is a proper virtual knotoid.

\section{Examples}
\label{sec:Example}

In this section we compute $\CPlusMinus$ and $\CHPlusMinus$ for three knotoids
shown in Fig~\ref{fig:10}.
In the table of knotoids published in~\cite{KorablevMayTarkaev}
they are $2_1$,$4_6$ and $5_{19}$.
The first two knotoids are interesting for us
because $\CPlusMinus(2_1)=\CPlusMinus(4_6)$
while $\CHPlus(2_1) \not=\CHPlus(4_6)$,
i.e., the pair of knotoids shows that $\CHPlusMinus$ is stronger than $\CPlusMinus$.
The third knotoid is an example of a knotoid
for which our invariants vanish,
thus they do not detect the triviality of a knotoid.

\begin{figure}[h!]
\centering 
       \begin{subfigure}[b]{0.2\textwidth}
                \centering
                \psfig{figure=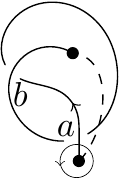, width=\textwidth}
                \caption{$2_1$}
                \label{fig:K2_1}
        \end{subfigure}%
       \begin{subfigure}[b]{0.3\textwidth}
                \centering
                \psfig{figure=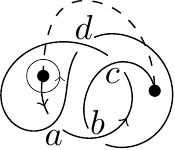, width=\textwidth}
                \caption{$4_6$}
                \label{fig:K4_6}
        \end{subfigure}%
       \begin{subfigure}[b]{0.3\textwidth}
                \centering
                \psfig{figure=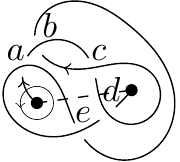, width=\textwidth}
\caption{$5_{19}$}
                \label{fig:K5_19}
        \end{subfigure}%
\caption{}
\label{fig:10}
\end{figure}

\subsection{The knotoid $2_1$}
\label{sec:K2_1}

The Gauss code of the knotoid $2_1$ (see Fig.~\ref{fig:K2_1}) is following:
$$\Plus{a}, \Minus{b}, \Minus{a},\Plus{b}.$$
The signs of crossings are:
$$
\begin{array}{ccc}
\text{Crossing} & a & b \\
\text{Sign} & 1 & 1
\end{array}
$$
There is the upper skew pair  and there is no lower skew pair:
$$\PairPlus{2_1} =\{(a,b)\}, \quad \PairMinus{2_1} =\emptyset,$$
Hence
\[
\begin{array}{lll}
\CPlus(2_1) &=\sign(a,b) &=1,\\
\CMinus(2_1) &=0. &\end{array}
\]

To compute values of $\CHPlusMinus$ of knotoids under consideration
we regard them as knotoids in the annulus $A$
and write resulting values in the form
described in the beginning of Section~\ref{sec:InS^2}.

Let $g \in H_1(A)$ is the homology class of the axial line of $A$
provided with counterclockwise orientation
(in Fig.~\ref{fig:10} the orientation is regarded with respect to the beginning of the diagram).
To determine the homology class of a loop in $A$
we compute the algebraic intersection number of the loop with
a simple arc $s$ connecting the endpoints of the diagram
(in Fig.~\ref{fig:10} it is depicted by dashed line)
and oriented from the end to the beginning.

For the knotoid under consideration we have
$[l(a]=[l(b)] =g$.
Hence
\[
\begin{array}{lll}
\CHPlus(2_1) &=\sign(a,b) H(a,b) &=\Angle{1},\\
 \CHMinus(2_1) &=0. &
\end{array}
\]

{\bf Remark.}
Using the knotoid $2_1$ we can show
that the classical construction of the Casson knot invariant does not give
an invariant in the case of knots in a thickened surface.
To this end we consider the closure of the knotoid
(about the closure map see, for example,\cite{KauffmanInvariants,TuraevKnotoids}),
i.e., consider the knotoid as a knotoid in the annulus
and identify boundary components of the annulus 
by the map carrying the endpoints of the knotoid one into another.
Resulting diagram lies in the torus and represents so-called ``virtual trefoil''
--- the simplest knot in the thickened torus.
The Gauss code corresponding to the diagram coincides with the Gauss code of the knotoid $2_1$
written in the beginning of this section.
If we place the base point just before $\Plus{a}$ then we have one upper skew pair.
But if we place the base point just after $\Plus{a}$ then
we obtain
$\Minus{b}, \Minus{a},\Plus{b}, \Plus{a}$
where no skew pair exists.

\subsection{The knotoid $4_6$}
\label{sec:K4_6}

The Gauss code of the knotoid $4_6$ (see Fig.~\ref{fig:K4_6})  is following:

$$\Minus{a}, \Plus{b}, \Minus{c}, \Plus{d}, \Plus{a}, \Minus{d}, \Minus{b}, \Plus{c}.$$
The signs of the crossings are:
\[
\begin{array}{ccccc}
\text{Crossing} &a&b&c&d\\
\text{Sign} &-1&1&1&-1.
\end{array}
\]

There are one upper skew pair and two lower skew pairs:
\[
\begin{array}{llll}
\PairPlus{4_6} &=\{(b,c)\}, & \sign(b,c)=1, & \\
\PairMinus{4_6} &=\{(a,b),(a,d)\}, & \sign(a,b)=-1, & \sign(a,d)=1.
\end{array}
\]
Therefore,
\[
\begin{array}{llll}
\CPlus(4_6) &=\sign(b,c) =1, & & \\
\CMinus(4_6) &=\sign(a,b)+\sign(a,d) &=-1+1 &=0.
\end{array}
\]

For the knotoid
$[l(a)] =[l(d)] =g$
and
$[l(b)] =[l(c)] =2g$.
\\
Hence
$H(a,b)$ and $H(a,d)$ coincide with $H_1(A)$
while $H(b,c)$ is the subgroup generated by $2 \cdot g$.
Therefore,
\[
\begin{array}{llll}
\CHPlus(4_6) &=\sign(b,c)H(b,c) =\Angle{2}, & & \\
 \CHMinus(4_6) &=\sign(a,b)H(a,b) +\sign(a,d)H(a,d) &=-\Angle{1}+\Angle{1} &=0.
\end{array}
\]

\subsection{The knotoid $5_{19}$}
\label{sec:K5_19}

The Gauss code of the knotoid $5_{19}$ (see Fig.~\ref{fig:K5_19}) is following:
$$\Minus{a}, \Plus{b}, \Minus{c}, \Plus{d}, \Plus{c}, \Minus{b}, \Minus{e}, \Plus{a}, \Plus{e}, \Minus{d}.$$
The signs of crossings are:
\[
\begin{array}{cccccc}
\text{Crossing} &a&b&c&d&e \\
\text{sign} &-1&1&1&1&1.
\end{array}
\]
There is no upper skew pair and there are two lower skew pairs:
\[
\begin{array}{llll}
\PairPlus{5_{19}} &=\emptyset & & \\
\PairMinus{5_{19}} &=\{(a,d),(c,d)\}, &
\sign(a,d)=-1, & \sign(c,d)=1.
\end{array}
\]
Hence
\[
\begin{array}{llll}
\CPlus(5_{19}) &=0, & & \\
\CMinus(5_{19}) &=\sign(a,d) +\sign(c,d) &=-1+1 &=0.$$
\end{array}
\]
The homology classes involved in existing pairs are following:
$$[l(a)]=0,\quad [l(c)] =-g, \quad [l(d)] =g.$$
Hence both $H(a,d)$ and $H(c,d)$ coincide with $H_1(A)$.
Hence
\[
\begin{array}{llll}
\CHPlus(5_{19}) &=0, & & \\
\CHMinus(5_{19}) &=\sign(a,d)H(a,d) +\sign(c,d)H(c,d) &=-\Angle{1}+\Angle{1} &=0.$$
\end{array}
\]

\section{Values of the invariants of tabulated knotoids}
\label{sec:Table}

We compute $\CPlusMinus$ and $\CHPlusMinus$ for all knotoids
listed in the table of prime proper knotoids
having diagrams with at most $5$ crossings
(see~\cite{KorablevMayTarkaev}).

\begin{tabular}{l|rr|rr}
\hline\noalign{\smallskip}
$\#$& $\CPlus$ & $\CMinus$ & $\CHPlus$ & $\CHMinus$ \\
\noalign{\smallskip}\hline\noalign{\smallskip}
$2_1$ & $1$ & $0$ & $\Angle{1}$ & $0$ \\
$3_1$ & $-1$ & $0$ & $-\Angle{1}$ & $0$ \\
$4_1$ & $1$ & $3$ & $\Angle{1}$ & $3\Angle{1}$ \\
$4_2$ & $1$ & $2$ & $\Angle{1}$ & $2\Angle{1}$ \\
$4_3$ & $2$ & $0$ & $2\Angle{1}$ & $0$ \\
$4_4 $ & $1$ & $-1$ & $\Angle{2}$ & $-\Angle{1}$ \\
$4_5$ & $3$ & $0$ & $2\Angle{1}+\Angle{2}$ & $0$ \\
$4_6  $ & $1$ & $0$ & $\Angle{2}$ & $0$ \\
$4_7$ & $2$ & $0$ & $\Angle{1}+\Angle{2}$ & $0$ \\
$5_1$ & $0$ & $-2$ & $0$ & $-2\Angle{1}$ \\
$5_2$ & $2$ & $3$ & $\Angle{0}+\Angle{1}$ & $\Angle{0}+2\Angle{1}$ \\
$5_3$ & $0$ & $-1$ & $\Angle{0}-\Angle{1}$ & $\Angle{0}-2\Angle{1}$ \\
$5_4$ & $-1$ & $-2$ & $-\Angle{1}$ & $-2\Angle{1}$ \\
$5_5$ & $1$ & $-1$ & $\Angle{1}$ & $-\Angle{1}$ \\
$5_6$ & $1$ & $0$ & $\Angle{0}$ & $\Angle{0}-\Angle{1}$ \\
$5_7$ & $2$ & $2$ & $2\Angle{1}$ & $2\Angle{1}$ \\
$5_8$ & $1$ & $2$ & $\Angle{0}$ & $\Angle{0}+\Angle{1}$ \\
$5_9$ & $1$ & $0$ & $\Angle{0}$ & $\Angle{0}-\Angle{1}$ \\
$5_{10}$ & $-1$ & $-1$ & $-\Angle{1}$ & $-\Angle{1}$ \\
$5_{11}$ & $1$ & $1$ & $\Angle{1}$ & $\Angle{1}$ \\
$5_{12}$ & $-2$ & $0$ & $-\Angle{1}-\Angle{2}$ & $0$ \\
$5_{13}$ & $0$ & $-1$ & $0$ & $-\Angle{2}$ \\
$5_{14}$ & $-1$ & $2$ & $-\Angle{2}$ & $2\Angle{1}$ \\
$5_{15}$ & $-1$ & $1$ & $-\Angle{2}$ & $\Angle{1}$ \\
$5_{16}$ & $1$ & $-2$ & $\Angle{1}$ & $-\Angle{1}-\Angle{2}$ \\
$5_{17}$ & $0$ & $-1$ & $-\Angle{1}+\Angle{2}$ & $-\Angle{1}$ \\
$5_{18}$ & $0$ & $1$ & $0$ & $\Angle{1}$ \\
$5_{19}$ & $0$ & $0$ & $0$ & $0$ \\
$5_{20}$ & $-2$ & $0$ & $-2\Angle{1}$ & $0$ \\
$5_{21}$ & $-1$ & $0$ & $-\Angle{1}$ & $0$ \\
$5_{22}$ & $-1$ & $-1$ & $-\Angle{1}$ & $-\Angle{1}$ \\
\noalign{\smallskip}\hline
\end{tabular}

Now we make afew remarks concerning the table.

{\bf 1.}
In\cite{KorablevMayTarkaev}
knotoids diagrams are regarded up to Reidemeister moves, ambient isotopies,
 mirror reflection and simultaneous switching of all crossing.
The latter transformation changes values of the invariants under consideration:
it permutes $\CPlus$ with $\CMinus$ and $\CHPlus$ with $\CHMinus$.
Hence values in the table above
should be regarded up to such permutations also.

{\bf 2.}
Taking into account the remark above we see that
the table contains $5$ pairs of knotoids having diagrams with coinciding values of all $4$ invariants:
$2_1$ and $5_{18}$,
$3_1$ and $5_{21}$,
$5_1$ and $5_{20}$,
$5_6$ and $5_9$,
$5_{10}$ and $5_{22}$.
Other $21$ knotoids in the table have pairwise distinct values  of $\CHPlusMinus$.

{\bf 3.}
It is clear, that $\CPlusMinus$ is weaker than $\CHPlusMinus$.
$\CPlusMinus$ divide $31$ knotoids under consideration
into $16$ subsets:
$3$ pairs, $2$ triples,
$2$ subsets consisting of $5$ knotoids
and $9$ knotoids have a unique values of $\CPlusMinus$.

{\bf 4.}
Theorem~\ref{theorem:Knot-type}
guarantees that all knotoids in the table except the only knotoid $5_{19}$
are proper knotoids
($5_{19}$ is proper also,
it is proved in~\cite{KorablevMayTarkaev}).
Besides, for $26$ of $31$ knotoids it is enough to use the first part of the theorem only.
Those $4$ knotoids for which the first part of Theorem~\ref{theorem:Knot-type} does not work while the second part do
are: $5_7, 5_{10}, 5_{11}, 5_{22}$.

{\bf 5.}
The table above contains exactly $2$ knotoids ($2_1$ and $4_1$)
for which the inequality~\eqref{eq:crn}
becomes equality.
These knotoids are the first and the second knotoids in the family
which we use in the proof of Theorem~\ref{theorem:crn}.
For knotoids having the crossing number $3$ and $5$
the right-hand side of~\eqref{eq:crn} is equal to $2$ and $6$, respectively,
while the maximal value of the left-hand side of~\eqref{eq:crn}
for knotoids in the table having crossing number $3$ and $5$
are $1$ and $5$, respectively.
The fact suggest following conjecture:
if a knotoid $K$ has odd crossing number then
$\|\CHPlus(K)\|+\|\CHMinus(K)\| +1 \leq [\frac{\crn(K)^2}4]$.

%\bibliography{Casson}   % name your BibTeX data base
%
%\end{document}

\end{document}